\documentclass{siamart251216}

\usepackage{graphicx} 
\usepackage{amsfonts,amsmath,amssymb}
\usepackage{booktabs,array,multirow}
\usepackage{cancel}
\usepackage[nocompress]{cite}
\usepackage{xcolor}
\usepackage{hyperref,url}
\usepackage{siunitx} 

\newsiamremark{remark}{Remark}
\newsiamremark{example}{Example}
\newsiamthm{assumption}{Assumption}
\newsiamthm{conjecture}{Conjecture}

\DeclareMathOperator{\spn}{span}

\newcommand{\rev}[1]{#1}

\definecolor{brilliantrose}{rgb}{1.0, 0.33, 0.64}
\definecolor{myviolet}{rgb}{0.21, 0.0, 0.85}
\definecolor{amethyst}{rgb}{0.6, 0.4, 0.8}
\definecolor{carrotorange}{rgb}{0.93, 0.57, 0.13}

\newenvironment{lauri}{\begin{quote}\color{brilliantrose} \small\sf Lauri $\clubsuit$~}{\end{quote}}
\newenvironment{stefan}{\begin{quote}\color{amethyst} \small\sf Stefan $\spadesuit$~}{\end{quote}}

\usepackage{tikz, pgfplots, mathtools, tikz-cd}

\def\C{\mathbb{C}}
\def\N{\mathbb{N}}

\def\Z{\mathbb{Z}}

\title{Flexible GMRES converges in two phases} 

\author{Stefan G\"{u}ttel\thanks{Department of Mathematics, The University of Manchester, Oxford Road, Manchester, M13\,9PL, United Kingdom, \texttt{stefan.guettel@manchester.ac.uk}}  \and
Lauri Nyman\thanks{Department of Mathematics, The University of Manchester, Oxford Road, Manchester, M13\,9PL, United Kingdom, \texttt{lauri.nyman@manchester.ac.uk} (corresponding author)}}

\begin{document}

\maketitle

\begin{abstract}
    We derive a sharp upper bound on the residuals produced by the flexible GMRES (FGMRES) method. The bound shows that FGMRES exhibits two phases of convergence depending on the residual tolerance of the inner preconditioner. For small tolerances, the convergence of FGMRES is practically geometric with a constant rate throughout, while for looser tolerances the two-phase behavior becomes more pronounced. We also show that the derived bound cannot be improved and construct an example for which it becomes an equality. 
\end{abstract}



\pagestyle{myheadings}
\thispagestyle{plain}
\markboth{S.~G\"{U}TTEL AND L.~NYMAN}{FLEXIBLE GMRES}

\section{Introduction}
The flexible GMRES (FGMRES) algorithm~\cite{saad1993flexible} is a widely used iterative solver for large-scale linear systems $A x = b$, \rev{with a matrix $A\in\C^{N\times N}$ and a vector $b\in\C^N$,} allowing for variable preconditioning.
One of the benefits of such a flexible method is that the preconditioner itself may be an iterative method that solves a linear system only approximately. This makes FGMRES particularly attractive in situations where the matrix $A$ is only (easily) accessible through matrix-vector products, including the important class of (Jacobian-free) Newton--Krylov methods for solving nonlinear equations~\cite{knoll2004jacobian}. \mbox{FGMRES} has become a standard tool in scientific computing, discussed in many books and reviews such as~\cite{benzi2002preconditioning,Saad2003,simoncini2007recent} and implemented in high-performance computing libraries such as PETSc~\cite{balay2019petsc} and Trilinos/Belos~\cite{bavier2012amesos2}. 
Some of the scientific problems where FGMRES has been used recently include the chemo-mechanical modeling of solid-state lithium batteries~\cite{liu2024role}, 
solving the pressure Poisson equation in fluids modeling~\cite{ge2007numerical}, the modeling of offshore wind turbines~\cite{li2022onset}, and in materials engineering~\cite{liu2018integrated}, to name just a few. Also recently, FGMRES has seen renewed theoretical interest due to the emergence of randomized variants~\cite{jang2025randomized,guttel2025stabilizing}.

Despite widespread use of the FGMRES algorithm, surprisingly little is known about its convergence. In particular, it is largely unclear how the residual tolerance of the preconditioner affects the convergence of the overall FGMRES method. The aim of this paper is to fill this gap.

To fix notation, we list  in Figure~\ref{fig:fgmres} the  FGMRES algorithm from~\cite{saad1993flexible}, as well as the FFOM method, which extends the FOM method \cite[Sec.~6.4]{saad2011numerical} to variable preconditioning; see also \cite{vuik1995new,guttel2025stabilizing}. The preconditioner $M_j$ in Figure~\ref{fig:fgmres} refers to any linear operator (matrix or black box solver) that approximately solves $A z_j \approx v_j$ at the $j$-th outer FGMRES iteration. Associated with this preconditioner is the residual
\[
    r_j^{\mathrm{P}} := v_j - A z_j.
\]
Throughout the paper we will make the simple assumption that
\begin{equation}\label{eq:assumemu}
    \|r_j^{\mathrm{P}}\| = \|v_j - A z_j\| \leq \mu
\end{equation}
for some $\mu < 1$ independent of $j$, where $z_j := M_j^{-1} v_j$. This is a standard recommendation for using FGMRES; see, e.g., \cite[Sec.~10]{simoncini2007recent}. 

One of our main results is a sharp bound showing that \emph{FGMRES exhibits two phases of convergence} dependent on the parameter~$\mu$.  If $\|r_j^{\mathrm{P}}\| \leq \mu \leq 1/2$ for all iterations~$j=1,\ldots,m$, then
\begin{equation}\label{eq:bndmain}
    \left\|r_m^{\mathrm{FG}}\right\| \leq \left\|r_0\right\| \frac{\mu^{m/2}}{\sqrt{U_m\left(\frac{1}{2\mu}\right)}}.
\end{equation}
Here, $r_j^{\mathrm{FG}} =  b - A x_j^{\mathrm{FG}}$ denotes the residual of the $j$-th FGMRES iterate~$x_j^{\mathrm{FG}}$ and $\|\cdot\|$ is the Euclidean norm. The function $U_m$ is the degree~$m$ Chebyshev polynomial of the second kind. 
Our new bound~\eqref{eq:bndmain} implies that FGMRES is guaranteed to converge without breakdown or stagnation, and it perfectly captures two convergence phases.
\smallskip
\begin{itemize}
    \item Phase~1: Initially, for small $m$, FGMRES  converges at a geometric rate close to $\mu$, the residual tolerance of the preconditioner. That is, $\left\|{r}_m^{\mathrm{FG}}\right\| \lessapprox \mu \left\|{r}_{m-1}^{\mathrm{FG}}\right\|$.
    \item Phase~2: As $m$ increases, the convergence rate gradually deteriorates toward~$\nu = \sin(\frac{1}{2} \arcsin(2\mu))$. That is, $\left\|{r}_m^{\mathrm{FG}}\right\| \lessapprox \nu \left\|{r}_{m-1}^{\mathrm{FG}}\right\|$.
\end{itemize}
\smallskip

A Taylor expansion of $\nu$ around $\mu=0$ gives $\nu = \mu + \mu^3/2 + O(\mu^5)$, which shows that for $\mu \ll 1/2$ FGMRES will converge at rate close to $\mu$ throughout. However, when $\mu\approx 1/2$ the rate eventually deteriorates to $\nu \approx \sqrt{\mu}$. See Table~\ref{tab:mu}, and Figure~\ref{fig:fgmres_sharp} for an illustration of this two-phase behavior. \rev{We emphasize that these observations relate to a sharp upper bound; how closely the true residuals follow this bound differs case-by-case. }

\begin{table}
\label{tab:mu}
\centering
\caption{The two phases of FGMRES convergence, dependent on the residual tolerance $\mu<1$ of the preconditioner. For $\mu\leq 1/2$, FGMRES first converges geometrically at rate $\mu$ and then deteriorates to rate $\nu = \sin(\frac{1}{2} \arcsin(2\mu))$. For $\mu\approx 0.5$, this deterioration is noticeable while for smaller $\mu$ the rate remains practically constant. For $\mu>1/2$, FGMRES may stagnate after an initial phase of residual reduction.}
\small
\label{table1}
\setlength{\tabcolsep}{10pt}{
\begin{tabular}{l l}             
  $\mu$ (phase 1) & $\nu$ (phase 2) \\
  \hline
  $0.01$ & $0.010000500087521$ \\
  $0.1$ & $0.100508962005208$ \\
$0.2$ & $0.204309643689220$ \\
  $0.3$ & $0.316227766016838$ \\
  $0.4$ & $0.447213595499958$ \\
  $0.5$ & $0.707106781186547 = \sqrt{\mu}$  \\
  \hspace*{-3.45mm}$>0.5$ & $1$ (stagnation)
\end{tabular}
}
\end{table}

Surprisingly, the bound~\eqref{eq:bndmain} cannot be improved in the sense that for a given linear system $Ax = b$ there always exists a sequence of preconditioners~$(M_j)_j$ producing residuals such that $\|r_j^{\mathrm{P}}\| \leq \mu$ and equality holds in~\eqref{eq:bndmain}. Furthermore, we show that one can construct a linear system $Ax = b$ for which equality holds in~\eqref{eq:bndmain} assuming GMRES($k$) as the preconditioner and~$\|r_j^{\mathrm{P}}\| \leq \mu$. (\rev{The example in Figure~\ref{fig:fgmres_sharp} has been produced that way.})

We also analyze the case $1/2 < \mu < 1$. In this case, FGMRES convergence may stagnate or even break down, after an initial phase of residual reduction. Hence, also in this case, FGMRES convergence is governed by two phases. Our stagnation result is also sharp in the sense that one can construct examples for which stagnation (or breakdown) occurs, and even predict the iteration index of when this happens.

The rest of the paper is organized as follows. In section~\ref{sec:review} we review the existing results on FGMRES convergence and explain the main difficulties in its analysis. Section~\ref{sec:new} derives the new results, including the a priori bound~\eqref{eq:bndmain}. In section~\ref{sec:sharp} we discuss different settings in which the bounds are sharp.
In section~\ref{sec:numex} we provide some numerical demonstrations before we conclude in section~\ref{sec:concl}.

\subsection*{Notation}

Throughout this paper, we adopt the following abbreviations that we use in the superscripts of the residual vector $r_m$ (at step $m$) to refer to the residuals of various Krylov subspace methods:
\smallskip
\begin{itemize}
    \item $r_m^{\mathrm{G}}$: \,\,Residual of the standard GMRES method
    \item $r_m^{\mathrm{F}}$: \,\,Residual of the full orthogonalization method (FOM)
    \item $r_m^{\mathrm{FF}}$: Residual of  flexible FOM (FFOM)
    \item $r_m^{\mathrm{FG}}$: Residual of  flexible GMRES (FGMRES)
    \item $r_m^{\mathrm{P}}$: \,\,Residual of the preconditioner at iteration~$m$ (i.e., $r_m^{\mathrm{P}} = v_m - A z_m$ $\rev{= (I-A M_m^{-1})v_m}$)
\end{itemize}
\smallskip
In addition, we write GMRES($k$) to denote that GMRES runs for exactly $k$ iterations. Moreover, FGMRES-GMRES($k$) denotes FGMRES preconditioned with GMRES($k$).

\begin{figure}[ht]
    \centering
    \begin{minipage}{1\textwidth}
        \hrule \vspace{2mm}
        \textbf{Input:} Matrix $A\in\C^{N\times N}$, vector $b\in\C^N$, initial guess $x_0\in\C^N$, $m\geq 1$ \\[1mm]
        $r_0 := b - Ax_0$; $\beta := \| r_0\|$; $v_1 := r_0/\beta$ \\[1mm]
        For $j=1,\ldots,m$
        \begin{itemize}
            \item[] $z_j := M_j^{-1} v_j$ \hfill \textit{\small (Preconditioner approximately solves $A z_j\approx v_j$)}
            \item[] $w := A z_j$
            \item[] For $i = 1,\dots,j$
                  \begin{itemize}
                      \item[] $h_{i,j} := v_i^* w$; \quad $w :=  w - h_{i,j}v_i$
                  \end{itemize}
            \item[] $h_{j+1,j} := \| w\|$; \quad $v_{j+1} := w / h_{j+1,j}$
        \end{itemize}
        \smallskip
        Define $Z_{m} := [ z_1,z_2,\ldots,z_{m} ]$ and let $\underline{H_m} \in \C^{(m+1)\times m}$ be the upper Hessenberg matrix of coefficients $h_{i,j}$. \\[2mm]
        \textbf{FGMRES:} Compute $x_m^{\mathrm{FG}}:= x_0 + Z_m y_m$, where $y_m$ minimizes $\|\beta e_1 - \underline{H_m} y \|$. That is, the residual norm $\|r_m^{\mathrm{FG}}\| = \| b - A x_m^{\mathrm{FG}} \|$ is minimized over all $x \in x_0 + \spn(Z_m)$. \\[1mm]
        \textbf{FFOM:} Compute $x_m^{\mathrm{FF}} := x_0 + Z_m y_m^{\mathrm{FF}}$, where $y_m^{\mathrm{FF}} = H_m^{-1}\beta e_1$, with $H_m \in \C^{m \times m}$ being the upper square block of $\underline{H_m}$.
        \vspace{2mm} \hrule
    \end{minipage}
    \caption{Pseudocode for the FGMRES and FFOM algorithms.}
    \label{fig:fgmres}
\end{figure}

\section{Existing convergence results for FGMRES}
\label{sec:review}

Analyzing the convergence of flexible GMRES is difficult because the residual loses its standard polynomial structure. In standard GMRES, the search space is a Krylov subspace, which means that the residual at step $m$ can be written as $r_m^{\mathrm{G}} = p(A)r_0$ for some polynomial $p \in \mathcal{P}_m$ with $p(0)=1$, where $\mathcal{P}_m$ denotes the set of polynomials of degree at most $m$.
This polynomial representation is exploited for standard GMRES convergence bounds. For example, Saad and Schultz \cite{SaadSchultz1986} showed in 1986 that if the matrix $A$ is diagonalizable as $A = V \Lambda V^{-1}$, the relative residual is bounded by a polynomial min-max problem:
$$ \frac{\|r_m^{\mathrm{G}}\|}{\|r_0\|} \leq \kappa(V) \min_{\substack{p \in \mathcal{P}_m \\ p(0)=1}} \max_{\lambda \in \sigma(A)} |p(\lambda)|. $$
There also exist geometric convergence bounds, such as Elman's bound \cite{Elman1982} for matrices with a positive definite symmetric part, and its refinements by Beckermann, Goreinov, and Tyrtyshnikov \cite{beckermann2005some}.

In FGMRES, however, the search space is constructed by using variable preconditioners. The residual takes the form
$$ r_m^{\mathrm{FG}} = b - A Z_m y_m, $$
where the columns of $Z_m$ depend on the varying preconditioners $M_j$ applied at each iteration~$j$. Because $Z_m$ does not span a standard Krylov subspace \rev{for the matrix $A$}, the standard min-max and geometric bounds for GMRES cannot be applied. With nonpolynomial preconditioners, it may not even be possible to write $r_m^{\mathrm{FG}} = p(A)r_0$. In addition, the following observations highlight the counterintuitive nature of FGMRES convergence. 
\begin{itemize}
    \item The FGMRES-GMRES($k$) residual can be larger than the corresponding residual coming from FGMRES with a fixed polynomial preconditioner of degree $k-1$. This is despite the fact that GMRES($k$) picks the optimal polynomial among those with degree $\leq k-1$, in the sense that $\| r^{\mathrm{P}} \|$ is minimized.
    \item Similarly, the FGMRES-GMRES($k$) residual can be larger than the FGMRES-GMRES($\ell$) residual for $\ell < k$, despite the inner preconditioner performing more algorithmic work to reach a smaller value for $\| r^{\mathrm{P}} \|$.
    \item In both of the above cases, the relative difference between the two residuals can become arbitrarily large.
\end{itemize}
\smallskip

The special case of using FGMRES with a fixed preconditioner $M$ is mathematically equivalent to applying standard GMRES with the matrix $B:=A \rev{M^{-1}}$. The most difficult case is indeed analyzing FGMRES convergence with variable preconditioning.

\rev{Although the span of $Z_m$ is not, in general, a Krylov subspace generated by the matrix $A$, a Krylov subspace interpretation is still possible. Indeed, when $Z_m$ is known, one can always construct a matrix $S$ a posteriori such that the span of $Z_m$ is a Krylov subspace generated by $S$. This representation alone, however, does not provide useful convergence bounds without additional control of $S$. Notay discusses this in \cite{Notay} in the context of the flexible conjugate gradient method, where he derives bounds on this a posteriori matrix $S$. 
He notes that these bounds are too pessimistic to yield useful convergence bounds without imposing overly stringent accuracy requirements on the preconditioning step. These results do not directly extend to the general nonsymmetric FGMRES setting considered here.
}

    
Among the works discussing FGMRES theoretically is~\cite{simoncini2002flexible}. 
{This work shows that if one uses GMRES as the preconditioner, the space spanned by the union of all Krylov vectors generated by the preconditioners (denoted $\mathcal{B}_k$ in that paper) keeps growing in dimension. 
Another discussion of FGMRES can be found in \cite{simoncini2003theory}. Therein, the viewpoint is to interpret FGMRES as an inexact Krylov subspace method. We start with a preconditioned system $A\rev{M^{-1}(Mx)} = b$ involving a fixed preconditioner~$M$. However, instead of ``knowing'' the preconditioner $M$, in FGMRES we apply variable $M_j$ that approximate $M$. If the approximation error $\epsilon$ between $M_j$ and $M$ is sufficiently small at every iteration, then we expect to get essentially the same convergence as for GMRES on the preconditioned system $A\rev{M^{-1}}$ plus an error term $O(\epsilon)$. For example, if exact GMRES for $A\rev{M^{-1}(Mx)}  = b$ converges geometrically at rate $\rho$, then in the inexact case we should expect a residual decrease
$$\|r_m\| \le C \rho^m \|r_0\| + O(\epsilon).$$
Here, the contraction factor $\rho$ depends on the spectral properties of the preconditioned matrix $A\rev{M^{-1}}$. While we used a fixed approximation error $\epsilon$ for simplicity, the authors of  \cite{simoncini2003theory} show that the approximation error $\epsilon_j$ at iteration $j$ is allowed to increase as the residual decreases, without negatively affecting the convergence behaviour. \rev{A similar approach has also been used in the context of FOM for parabolic optimal control problems \cite{DuSarkisSchaererSzyld2013}.}

For the purposes of convergence analysis, interpreting FGMRES as an inexact Krylov method is somewhat restrictive. It requires quantifying how closely the preconditioners $M_j$ relate to $M$, and making sure that this difference is small enough. Additionally, the $O(\epsilon)$ term is rather impractical as it suggests that FGMRES may never be able to reduce the residual to zero, which is not usually the case. Hence, the convergence bounds resulting from inexact Krylov theory are likely of limited utility.

A very simple  bound on the FGMRES residual $r_m^{\mathrm{FG}}$ at iteration $m$ in terms of the FFOM residual at the previous iteration is
\begin{align} \label{eq:FGMRESFFOMbound}
    \big\| r_m^{\mathrm{FG}} \big\| \leq
    \big\| r_{m-1}^{\mathrm{FF}} \big\|\cdot
    \big\| r_{m}^{\mathrm{P}} \big\|.
\end{align}
This a posteriori bound has been derived in \cite{guttel2025stabilizing}, and it is also essentially contained in the proof of Lemma~3 by Vuik~\cite{vuik1995new}. It allows us to predict the FGMRES residual decrease at iteration~$m$ \emph{before} the preconditioner has completed. This is practically useful in the setup considered in~\cite{guttel2025stabilizing}, where it is recommended to run as preconditioner a GMRES variant with randomized sketching for as long as possible, hence making the preconditioner the main expense of the overall FGMRES iteration. We will also build on \eqref{eq:FGMRESFFOMbound} to derive most of our results.

\section{New convergence results}\label{sec:new}

We now derive our main convergence results.

\subsection{A recursive bound on the FGMRES residual}\label{sec:recur}
\rev{We first} extend a well-known equality by Brown \cite{Brown1991} relating the size of the FOM residual to that of the GMRES residual
\begin{align} \label{eq:FOM_GMRES}
    \left\|r_m^{\mathrm{F}}\right\|=\frac{\left\|r_m^{\mathrm{G}}\right\|}{\sqrt{1-\left(\left\|r_m^{\mathrm{G}}\right\| /\left\|r_{m-1}^{\mathrm{G}}\right\|\right)^2}}
\end{align}
to the flexible setting \rev{by using a result by Eiermann and Ernst \cite[Theorem~3.4]{EiermannErnst2001}}. This is stated formally in Proposition~\ref{prop:ffom_fgmres}. 


\begin{proposition} \label{prop:ffom_fgmres}
    For all $m \in \Z^+$ for which $r_m^{\mathrm{FF}}$ is defined (i.e., $H_m$ is invertible), it holds that
    \begin{align} \label{eq:FFOM_FGMRES}
        \left\|r_m^{\mathrm{FF}}\right\| = \frac{\left\|r_m^{\mathrm{FG}}\right\|}{\sqrt{1-\left(\left\|r_m^{\mathrm{FG}}\right\| /\left\|r_{m-1}^{\mathrm{FG}}\right\|\right)^2}}.
    \end{align}
\end{proposition}
\begin{proof}
\rev{Apply \cite[Theorem~3.4]{EiermannErnst2001} with $\mathcal W_m=\spn(AZ_m)$ and $\mathcal V_m=\spn(V_m)=\spn\{r_0\}+\mathcal W_{m-1}$. The MR and OR residuals are then the FGMRES and FFOM residuals,
respectively. Equations (3.14) and (3.16) in that reference give
\[
\|r_m^{\mathrm{FG}}\|=s_m\|r_{m-1}^{\mathrm{FG}}\|,
\qquad
\|r_m^{\mathrm{FG}}\|=c_m\|r_m^{\mathrm{FF}}\|,
\]
where $s_m$ and $c_m$ are the sine and cosine of the same angle.
Hence,
\[
c_m=\sqrt{1-s_m^2}
=\sqrt{1-\left(
\frac{\|r_m^{\mathrm{FG}}\|}{\|r_{m-1}^{\mathrm{FG}}\|}
\right)^2}.
\]
Substituting this into the second equality and dividing by $c_m$ yields the result.
}
\end{proof}

We now combine Proposition \ref{prop:ffom_fgmres} with the bound~\eqref{eq:FGMRESFFOMbound} to derive the following recursive bound on the FGMRES residual.

\begin{proposition} \label{prop:recurrence_gamma}
    It holds that
    \begin{align}
        \begin{aligned} \label{eq:successive_rG_gamma}
             & \left\|r_m^{\mathrm{FG}}\right\| \leq \gamma_m \left\|r_{m-1}^{\mathrm{FG}}\right\|, \\
             & \gamma_m := \frac{\|r_m^{\mathrm{P}}\|}{\sqrt{1-\gamma_{m-1}^2}}, \quad \gamma_0 := 0,
        \end{aligned}
    \end{align}
    for all $m \in \Z^+$ for which $\max_{j \leq m} \gamma_{j}
        < 1$.
\end{proposition}

\begin{proof}
    We prove the result by induction. In the base case $m=1$, we have $\gamma_1 = \|r_1^{\mathrm{P}}\|$ and $v_1 = r_0 / \left\|r_0\right\|$. \rev{By using \eqref{eq:FGMRESFFOMbound}, it follows that} $\left\|r_1^{\mathrm{FG}}\right\| \leq \gamma_1 \left\|r_0\right\|$. 

    Let us assume that the statement holds for $m = k \geq 1$ with $\max_{j \leq k} \gamma_j < 1$. Then, $$\left\|r_k^{\mathrm{FG}}\right\| / \left\|r_{k-1}^{\mathrm{FG}}\right\| \leq \gamma_k,$$
    and hence
    \begin{align*}
        \frac{\|r_{k+1}^{\mathrm{P}}\|}{\sqrt{1-(\left\|r_k^{\mathrm{FG}}\right\| / \left\|r_{k-1}^{\mathrm{FG}}\right\|)^2}} \leq \rev{\frac{\|r_{k+1}^{\mathrm{P}}\|}{\sqrt{1-\gamma_k^2}} = }  \gamma_{k+1}.
    \end{align*}
    Combining this with the bound \eqref{eq:FGMRESFFOMbound} and the equality \eqref{eq:FFOM_FGMRES} yields
    \begin{align*}
        \left\|r_{k+1}^{\mathrm{FG}}\right\| \leq \left\|r_k^{\mathrm{FF}}\right\|\|r_{k+1}^{\mathrm{P}}\| \leq \gamma_{k+1} \left\|r_k^{\mathrm{FG}}\right\|,
    \end{align*}
    which completes the proof.
\end{proof}

\subsection{An explicit bound on the FGMRES residual}
The assumption on the~$\gamma_j$ in Proposition~\ref{prop:recurrence_gamma}, namely that $\max_{j \leq m} \gamma_{j} < 1$, guarantees that the denominator in the expression for $\gamma_{j}$ is a positive real number for all $j \leq m$. If we assume 
\[
\left\|r_j^{\mathrm{P}}\right\| = \| v_j - A z_j \| \leq \mu \leq \frac{1}{2} \quad \text{for all $j$},
\]
we can derive a similar recursive bound without needing to assume $\max_{j \leq m} \gamma_{j} < 1$. The resulting bound is easier to analyze and allows us to express a bound for $\left\|r_m^{\mathrm{FG}}\right\|$ purely in terms of the contraction factor~$\mu$.
\begin{lemma} \label{prop:recurrence}
    Let $0 < \mu \leq \frac{1}{2}.$ We have 
    \begin{align}
        \begin{aligned} \label{eq:successive_rG}
             & \left\|r_m^{\mathrm{FG}}\right\| \leq \omega_m \left\|r_{m-1}^{\mathrm{FG}}\right\|, \\
             & \omega_m := \frac{\mu}{\sqrt{1-\omega_{m-1}^2}}, \quad \omega_0 := 0,
        \end{aligned}
    \end{align}
    for all $m \in \Z^+$. 
\end{lemma}
\begin{proof}
    First, we show that $\omega_{j} < 1$ for all $j$ whenever $0 < \mu \leq \frac{1}{2}.$ We do this by showing that if $\omega_j$ lies in the interval $\left(0, \frac{1}{\sqrt{2}}\right]$, so does $\omega_{j+1}$. Indeed, if $0 < \mu \leq \frac{1}{2}$ and $0 < \omega_{j} \leq \frac{1}{\sqrt{2}}$, then
    \begin{align*}
       \omega_{j+1} = \frac{\mu}{\sqrt{1-\omega_{m}^2}} \leq \sqrt{2} \mu \leq \frac{1}{\sqrt{2}}.
    \end{align*}
    This, combined with the fact that clearly $\omega_{j+1} > 0$, proves the result. The remaining proof of the lemma is similar to the proof of Proposition \ref{prop:recurrence_gamma}.
\end{proof}

The following two lemmas characterize the numbers $\omega_m$ appearing in Lemma~\ref{prop:recurrence}.

\begin{lemma} \label{lem:monotonic_omega}
 Let $0 < \mu \leq \frac{1}{2}.$ The sequence $(\omega_j)_{j \in \Z^+}$ is monotonically increasing.
\end{lemma}
\begin{proof}
    The sequence $(\omega_j)_{j \in \Z^+}$ of \eqref{eq:successive_rG} can be generated as $\omega_j = f(\omega_{j-1})$, where
    $$f(x) = \frac{\mu}{\sqrt{1-x^2}}.$$
    The result follows from the fact that $f$ is strictly increasing in the interval $x \in [0,1)$, while the base case $f(\omega_1) > f(\omega_0)$ is clearly satisfied.
\end{proof}
In the following lemma, we denote $\omega_m = \omega_m(\mu)$ to emphasize the dependence of the value for $\omega_m$ on the parameter $\mu$.
\begin{lemma} \label{lem:monotonic_in_mu}
    Let $m \in \Z^+$ and $0<\mu_1 < \mu_2 \leq \frac{1}{2}$. It holds that $\omega_m(\mu_1) < \omega_m(\mu_2)$.
\end{lemma}
\begin{proof}
    Differentiating $\omega_m$ of \eqref{eq:successive_rG} with respect to $\mu$ yields a linear recurrence for the derivatives:
    \[\frac{d\omega_m}{d\mu} = \frac{1}{\sqrt{1 - \omega_{m-1}^2}} + \frac{\mu \omega_{m-1}}{(1 - \omega_{m-1}^2)^{3/2}} \left( \frac{d\omega_{m-1}}{d\mu} \right).\]
    Now, it is easy to prove by induction that $\frac{d\omega_m}{d\mu} > 0$ in the interval $\mu \in (0, \frac{1}{2}]$. In the base step $m=1$, we have $\omega_1 = \mu$, so $\frac{d\omega_1}{d\mu} = 1 > 0$. In the inductive step, we assume that $\frac{d\omega_{m-1}}{d\mu} > 0$, which implies that both terms in the above expression for $\frac{d\omega_m}{d\mu}$ are positive for $m\geq2$, as $\omega_{m-1} \in (0,1)$. Because the derivative is strictly positive for all $m$ in the interval $\mu \in (0,\frac{1}{2}]$, the function $\omega_m(\mu)$ is strictly increasing in the interval $\mu \in (0,\frac{1}{2}]$.
\end{proof}

We are now in the position to state one of our main theorems.

\begin{theorem} \label{thm:explicit_bound}
    Let $0 < \mu \leq \frac{1}{2}$. For all $m<N$, the FGMRES residual at step $m$ satisfies
    \begin{align} \label{eq:cheb_bound}
        \left\|r_m^{\mathrm{FG}}\right\| & \leq \left\|r_0\right\| \frac{\mu^{m/2}}{\sqrt{U_m\left(\frac{1}{2\mu}\right)}},
    \end{align}
    where $U_m$ denotes the degree~$m$ Chebyshev polynomial of the second kind.
\end{theorem}
\begin{proof}
    Applying \eqref{eq:successive_rG} recursively leads to
    \begin{align} \label{eq:bound}
        \left\|r_m^{\mathrm{FG}}\right\| \leq \left\|r_0\right\| \prod_{j=1}^m \omega_j.
    \end{align}
    We seek to express $ \prod_{j=1}^m \omega_j$ purely in terms of the contraction factor $\mu$. To simplify the notation, define $x_m := \omega_m^2$ so that the recurrence \eqref{eq:successive_rG} becomes
    $$x_m = \frac{\mu^2}{1 - x_{m-1}}, \quad x_0 = 0.$$
    We approach this fractional recurrence relation by expressing it as a system of two linear recurrence relations. To this end, let $\{a_j\}_{j \in \N}, \{b_j\}_{j \in \N}$ denote sequences such that $x_m = \frac{a_m}{b_m}$. Since $x_0 = \omega_0^2 = 0$, we can set $a_0 = 0$ and $b_0 = 1$. Substituting $x_m = \frac{a_m}{b_m}$ into the recurrence relation gives
    \begin{align} \label{eq:recurrence}
        \frac{a_m}{b_m} = \frac{\mu^2}{1 - \frac{a_{m-1}}{b_{m-1}}} = \frac{\mu^2 b_{m-1}}{b_{m-1} - a_{m-1}}.
    \end{align}
    From this, we define the following system of two linear recurrence relations with initial conditions $a_0 = 0$ and $b_0 = 1$, whose solution clearly satisfies the relation \eqref{eq:recurrence}:
    \begin{align} \label{eq:recurrence_system}
        \begin{aligned}
            a_m & = \mu^2 b_{m-1},     \\
            b_m & = b_{m-1} - a_{m-1}.
        \end{aligned}
    \end{align}
    Now, $x_j = \frac{a_j}{b_j} = \frac{\mu^2 b_{j-1}}{b_j}$, which leads to convenient cancellations in the product $\prod_{j=1}^m x_j$:
    $$\prod_{j=1}^m x_j = \prod_{j=1}^m \left( \frac{\mu^2 b_{j-1}}{b_j} \right) = \mu^{2m} \left( \frac{b_0}{b_1} \cdot \frac{b_1}{b_2} \cdots \frac{b_{m-1}}{b_m} \right) = \mu^{2m} \frac{b_0}{b_m}.$$
    Since $b_0=1$, we have that $\prod_{j=1}^m \omega_j = \frac{\mu^m}{\sqrt{b_m}}.$
    Now, we only need an explicit formula for $b_m$. By substituting the first equation of \eqref{eq:recurrence_system} into the second, we get a single second-order linear recurrence relation for $b_m$:
    \begin{align} \label{eq:2nd_order_recurrence}
        b_m = b_{m-1} - \mu^2 b_{m-2}.
    \end{align}
    The characteristic equation for $b_m - b_{m-1} + \mu^2 b_{m-2} = 0$ is
    $$r^2 - r + \mu^2 = 0,$$
    whose roots are
    \begin{align*}
        r_{\pm} = \frac{1 \pm \sqrt{1 - 4\mu^2}}{2}.
    \end{align*}
    Hence, all solutions are of the form $b_m = \rev{c_1}(r_+)^m + \rev{c_2}(r_-)^m$, and using the initial conditions $b_0 = 1$ and $b_1 = 1$, we can solve for the coefficients \rev{$c_1, c_2$}. Standard algebraic manipulation yields
    \begin{align} \label{eq:bm}
        b_m = \frac{(r_+)^{m+1} - (r_-)^{m+1}}{r_+ - r_-}.
    \end{align}
    This lets us express the bound \eqref{eq:bound} as
    \begin{equation}
        \begin{aligned}
        \label{eq:bound_with_r}
             \left\|r_m^{\mathrm{FG}}\right\| & \leq \left\|r_0\right\| \mu^m \left(  \frac{r_+ - r_-}{(r_+)^{m+1} - (r_-)^{m+1}} \right)^{1/2}.                                                                                       \\
        \end{aligned}
    \end{equation}
    We also note that, after the change of variables $c_m = \frac{b_m}{\mu^m}, \ \mu = \frac{1}{2x}$, the recurrence relation \eqref{eq:2nd_order_recurrence} becomes that of the Chebyshev polynomials of the second kind, which are polynomials defined by 
    $$
    { U_{m}(\cos \theta )\sin \theta ={\sin }{\big (}(m+1)\theta {)}, \quad  0 \leq \theta \leq \pi.}
    $$ As such, we can express the bound \eqref{eq:bound_with_r} more compactly as 
    \begin{equation*} 
        \begin{aligned}
            \left\|r_m^{\mathrm{FG}}\right\| & \leq \left\|r_0\right\| \frac{\mu^{m/2}}{\sqrt{U_m\left(\frac{1}{2\mu}\right)}}.
        \end{aligned}
    \end{equation*}
    These bounds are valid at iteration $m$ if the sequence of contraction factors $(\omega_j)_j$ satisfies $\omega_j < 1$ for all $j < m$, and $\omega_m \leq 1$. Setting $\omega_m = 1$ yields $a_m = b_m$, which in turn implies that $b_{m+1}=0.$ This is attained precisely at the roots of $U_{m+1}$, which are well known to be ${x_{k}=\cos \left({\frac {k}{m+2}}\pi \right),\ k=1,\dots ,m+1.}$ Hence, the smallest positive value for $\mu$ for which $\omega_m = 1$ is
    \begin{equation}\label{eq:mum}
        \mu_m = \frac{1}{2 \cos\left(\frac{\pi}{m+2}\right)}.
    \end{equation}
    Lemma~\ref{lem:monotonic_in_mu} implies that, for $\mu < \mu_m$, it holds that $\omega_j < 1$ for all $j \leq m$. 
    Moreover, it holds that the sequence $(\mu_m)_m$ is monotonically decreasing with $\lim_{m \rightarrow \infty} \mu_m = \frac{1}{2}$. 
    It follows that the bound \eqref{eq:cheb_bound} holds for $\mu \in [0,\frac{1}{2}].$ This completes the proof.
\end{proof}


\subsection{Guaranteed contraction in each iteration}
The bound \eqref{eq:cheb_bound} characterizes the global convergence behaviour. While \eqref{eq:successive_rG_gamma} characterizes the local behaviour, an explicit formula for $\gamma_m$ is not available a priori. However, the simplified version of the local bound given in \eqref{eq:successive_rG} allows for this. 

\begin{corollary}\label{cor:localrate}
    In the notation of the proof of Theorem \ref{thm:explicit_bound}, using the fact that $a_m = \mu^2 b_{m-1}$ together with the observation that $b_{m} = \mu^{m} U_{m}\left(\frac{1}{2\mu}\right)$, we can deduce that
    $$\omega_m = \mu^{\frac{1}{2}} \sqrt{\frac{U_{m-1}\left(\frac{1}{2\mu}\right)}{U_{m}\left(\frac{1}{2\mu}\right)}}.$$
    The relationship between the residuals of successive iterations is then
    \begin{align}
        \begin{aligned}
             & \left\|r_m^{\mathrm{FG}}\right\| \leq \left\|r_{m-1}^{\mathrm{FG}}\right\| \mu^{\frac{1}{2}} \sqrt{\frac{U_{m-1}\left(\frac{1}{2\mu}\right)}{U_{m}\left(\frac{1}{2\mu}\right)}}. \\
        \end{aligned}
    \end{align}
\end{corollary}
The following proposition characterizes the limiting behavior of the sequence~$(\omega_m)$.

\begin{proposition} \label{prop:mu_range}
    Let $\mu \leq \frac{1}{2}.$ It holds that
\[
\lim_{m \to \infty} \omega_m = \sin(0.5 \arcsin(2\mu)) \leq \sqrt{\mu}.
\]
with
\[
 \lim_{\mu \to 1/2} \sin(0.5 \arcsin(2\mu)) = \frac{1}{\sqrt{2}}.
\]
\end{proposition}
\begin{proof}
    For $\mu \leq \frac{1}{2}$, straightforward algebraic manipulation \rev{of  \eqref{eq:bound_with_r}} yields
    $$\omega_* := \lim_{m \to \infty} \omega_m = \sqrt{r_-},$$ where $r_-$ is defined as in the proof of Theorem \ref{thm:explicit_bound}. The equality $$\sqrt{r_-} = \sin(0.5 \arcsin(2\mu))$$ follows from standard trigonometric identities. The observation
    \begin{align*}
        \omega_*^2 = r_- \leq \sqrt{r_- r_+} = \mu
    \end{align*}
    shows the upper bound of the first limit. The second limit follows from a direct evaluation.
\end{proof}

In particular, the fact that $\omega_m < 1$ for all $m$ guarantees that no breakdown can occur unless the residual drops to zero. This in turn guarantees that the FGMRES iteration converges in at most $N$ iterations.

\subsection{Bound on the FFOM residual} 
Due to the  relation between \mbox{FGMRES} and FFOM  in view of \eqref{eq:FFOM_FGMRES}, it is  possible to derive a convergence bound for FFOM.

\begin{theorem}
    Let $0 < \mu \leq \frac{1}{2}$. For all $m<N$, the FFOM residual satisfies
    \begin{equation} \label{eq:FFOM_full}
        \left\|r_m^{\mathrm{FF}}\right\| \leq \left\|r_0\right\| \frac{\mu^{(m-1)/2}}{\sqrt{U_{m+1}\left(\frac{1}{2\mu}\right)}}.
    \end{equation}
\end{theorem}
\begin{proof}
    Let $\omega_m, a_m, b_m$ be defined as in the proof of Theorem \ref{thm:explicit_bound}. We can derive this bound by first combining \eqref{eq:FFOM_FGMRES} and \eqref{eq:successive_rG} to yield
    \begin{align*}
        \left\|r_m^{\mathrm{FF}}\right\| & \leq \frac{\left\|r_m^{\mathrm{FG}}\right\|}{\sqrt{1-\omega_m^2}}.
    \end{align*}
    Recall that $\omega_m^2 = a_m / b_m$ and $b_{m+1} = b_m - a_m$. Hence, $$1 - \omega_m^2 = 1 - \frac{a_m}{b_m} = \frac{b_m - a_m}{b_m} = \frac{b_{m+1}}{b_m}.$$
    Combining this with the inequality $$\left\|r_m^{\mathrm{FG}}\right\| \leq \left\|r_0\right\| \frac{\mu^m}{\sqrt{b_m}}$$
    yields the inequality
    \begin{align*}
        \left\|r_m^{\mathrm{FF}}\right\| \leq \frac{\left\|r_m^{\mathrm{FG}}\right\|}{\sqrt{1-\omega_m^2}} \leq \left\|r_0\right\| \frac{\mu^m}{\sqrt{b_{m+1}}}.
    \end{align*}
    The \rev{result} follows from observing that $b_{m+1} = \mu^{m+1} U_{m+1}\left(\frac{1}{2\mu}\right)$, as per the proof of Theorem \ref{thm:explicit_bound}. 
\end{proof}

\subsection{Poor preconditioners}


The analysis of the previous subsections assumes the availability of good-quality preconditioners satisfying $\|r_j^{\mathrm{P}}\| \leq \mu \leq 1/2$. In this subsection, we extend the analysis to the case $\mu > 1/2$. In this regime, the $\nu$-value of Phase~2 documented in Table~\ref{tab:mu} becomes $\nu=1$. In other words, the convergence is characterized by two phases: initial convergence and stagnation.

Even when $\mu > 1/2$, the convergence behaviour is perfectly captured by \eqref{eq:cheb_bound} of Theorem \ref{thm:explicit_bound} as long as the local contraction factors satisfy $\omega_j \leq 1$ for all $j \leq m$. By inspecting the proof of Theorem \ref{thm:explicit_bound}, we see that this corresponds to the interval $0 < \mu \leq \mu_m$, where $\mu_m$ is defined as in \eqref{eq:mum}:
    \begin{equation*}
        \mu_m = \frac{1}{2 \cos\left(\frac{\pi}{m+2}\right)}.
    \end{equation*}
For $\mu=\mu_m$, the denominator of the bound \eqref{eq:cheb_bound} simplifies to 1, 
and the bound becomes
\[
    \left\|r_m^{\mathrm{FG}}\right\| \leq \left\|r_0\right\| (\mu_m)^{\frac{m}{2}}.
\]
This describes the total reduction in the residual that is possible to guarantee. Beyond iteration $m$, the convergence might come to a complete standstill.  
In the stagnation phase, linear independence of the vectors $z_j$ generated by FGMRES cannot be guaranteed anymore. A breakdown can happen already in the first iteration where $\omega_j\leq 1$ is not satisfied. 


We define the stalling index $m^*$ as the index after which the theoretical bound \eqref{eq:cheb_bound} stalls. This gives us the length of the convergence phase. 
Table \ref{tab:mum} illustrates that as $\mu$ increases beyond 0.5, the stalling index $m^*$ rapidly decreases. Already for $\mu = 0.55$, the initial convergence phase lasts only for 5 iterations, after which the bound stagnates at relative residual $\approx 0.20$. 
A numerical illustration of this phenomenon is given in Figure~\ref{fig:fgmres_stalling}.

Our analysis suggests that, if possible, one should prioritise forming preconditioners that strictly satisfy $\mu \leq 1/2$ as this guarantees convergence to zero residual with at least the rate $\sqrt{\mu}$. Even slight violations of this condition can be detrimental for convergence.   


\begin{table}[ht]
    \centering
    \caption{Stalling index $m^*$ and the corresponding relative residual bound \eqref{eq:cheb_bound} for different values for $\mu$.}
    \small
    \label{tab:mum}
    \setlength{\tabcolsep}{10pt}{
        \begin{tabular}{l l l}
            Contraction factor $\mu$ & Stalling index $m^*$ & Residual bound $\frac{\mu^{m^*/2}}{\sqrt{U_{m^*}\left(\frac{1}{2\mu}\right)}}$ \\
            \hline
            $0.5$ & $\infty$ & $0$ \\
            $0.501$ & $47$ & $6.75 \times 10^{-8}$ \\
            $0.51$ & $13$ & $9.34 \times 10^{-3}$ \\
            $0.55$ & $5$ & $1.98 \times 10^{-1}$ \\
            $0.6$ & $3$ & $4.08 \times 10^{-1}$ \\
            $0.8$ & $1$ & $8.00 \times 10^{-1}$ \\
        \end{tabular}
    }
\end{table}

\section{Sharpness results}\label{sec:sharp}

We now list several results that characterize different ways in which our  FGMRES convergence bound is sharp. The first theorem states that for any nonsingular linear system $Ax=b$ and $\mu\leq 1/2$, there always exists a sequence of worst-case preconditioners attaining the bound in Theorem~\ref{thm:explicit_bound}.

\begin{theorem} [Sharpness of the FGMRES bound] \label{thm:sharpness}
    For any positive integer $m < N$, any invertible matrix $A \in \C^{N \times N}$, any right-hand side $b \in \C^N \setminus \{0\}$, and any contraction factor $0 < \mu \leq \frac{1}{2}$, there exists a sequence of preconditioners such that, for all $1 \leq j \leq m$, it holds that $\|r_j^{\mathrm{P}}\| \leq \mu$ and $$\big\| r_j^{\mathrm{FG}} \big\| = \big\| r_{j-1}^{\mathrm{FF}} \big\| \cdot \big\| r_{j}^{\mathrm{P}} \big\|.$$ 
    That is, the bound \eqref{eq:FGMRESFFOMbound} holds as an equality. For this choice of preconditioners, the bound in \eqref{eq:successive_rG} and hence also in Theorem \ref{thm:explicit_bound} hold as equalities. Consequently, the bound in Theorem~\ref{thm:explicit_bound} is sharp.
\end{theorem}

\begin{proof}
    We prove the result by induction. The base case $k=1$ yields $$\big\| r_1^{\mathrm{FG}} \big\| = \big\| r_0 \big\| \cdot \big\| r_{1}^{\mathrm{P}} \big\|,$$ which clearly holds. From here on, we assume that the result holds for step $k-1$ and prove it for step $k < N$.

    Let $w_j := A z_j$ and correspondingly $W_j := A Z_j$. We define $\tilde x := x_{k-1}^{\mathrm{FF}} + c z_k$, where the scalar $c$ is chosen such that $r_{k-1}^{\mathrm{FF}} = c v_k$. Note that if $c=0$, the result of the theorem holds trivially. Therefore, from here on, we assume that $c \neq 0$. The corresponding residual is then $\tilde r = c r_k^{\mathrm{P}}.$ Clearly, $\|r_k^{\mathrm{FG}}\| \leq \|\tilde r\| = \|r_{k-1}^{\mathrm{FF}}\| \|r_k^{\mathrm{P}}\|$, as stated in bound~\eqref{eq:FGMRESFFOMbound}. Moreover, equality is equivalent to $\tilde r \perp \spn(W_k)$, which in turn is equivalent to $r_k^{\mathrm{P}} \perp \spn(W_k).$

    We construct $w_k \in \C^N$ satisfying $r_k^{\mathrm{P}} = v_k - w_k \perp \spn(W_k)$ as well as the contraction constraint $\|r_k^{\mathrm{P}}\| = \mu$. To this end, let $v_k = p_k + u_k$, where $p_k \in \spn(W_{k-1})$ and $u_k \perp \spn(W_{k-1})$. We seek $w_k$ in the form $$ w_k := p_k + y_k,$$ where $y_k \perp \spn(W_{k-1})$ is a vector to be determined.

    The inner preconditioner residual is then $r_k^{\mathrm{P}} = u_k - y_k$. By construction, $r_k^{\mathrm{P}} \perp \spn(W_{k-1})$. The conditions $r_k^{\mathrm{P}} \perp w_k$ and $\|r_k^{\mathrm{P}}\| = \mu$ become
    $$ \rev{(u_k - y_k)^* (p_k + y_k)} = 0 \quad \text{and} \quad \|u_k - y_k\|^2 = \mu^2.$$
    Noting that cross-terms vanish because $p_k \in \spn(W_{k-1})$ and $u_k, y_k \perp \spn(W_{k-1})$, the orthogonality condition simplifies to $\rev{u_k^* y_k} = \|y_k\|^2$.

    Because $k < N$, we have $\dim \spn(W_{k-1})^\perp \geq 2$. Let $e_1 = u_k / \|u_k\|$. We can pick a unit vector $e_2 \in \spn(W_{k-1})^\perp$ such that $e_2 \perp e_1$. We let $y_k$ be a real linear combination of these such that $y_k = \alpha e_1 + \beta e_2$, in which case the condition $\rev{u_k^* y_k} = \|y_k\|^2$ becomes $$ \alpha \|u_k\| = \alpha^2 + \beta^2, $$ and the contraction constraint $\|u_k - y_k\|^2 = \mu^2$ becomes
    $$\|u_k\|^2 - 2\alpha\|u_k\| + (\alpha^2 + \beta^2) = \mu^2.$$
    Substituting in $\alpha^2 + \beta^2 = \alpha \|u_k\|$ yields
    $$ \alpha = \frac{\|u_k\|^2 - \mu^2}{\|u_k\|}. $$
    Substituting this back in to find $\beta$ yields
    $$ \beta^2 = \mu^2 \frac{\|u_k\|^2 - \mu^2}{\|u_k\|^2}, $$
    and we can simply select
    $$ \beta = \frac{\mu}{\|u_k\|} \sqrt{\|u_k\|^2 - \mu^2}. $$
    This yields a real value for $\beta$ if and only if $\mu \leq \|u_k\|$.

    Finally, we show that the condition $\mu \leq \|u_k\|$ is always satisfied under the initial assumption $\mu \leq \frac{1}{2}$. By definition, the FGMRES residual $r_{k-1}^{\mathrm{FG}}$ is the orthogonal projection of $b$ onto $\spn(W_{k-1})^\perp$. Projecting the FFOM residual equation $b = c v_k + W_{k-1} y^{\mathrm{FF}}$ onto $\spn(W_{k-1})^\perp$ yields $r_{k-1}^{\mathrm{FG}} = c(v_k - p_k) = c u_k$. Recalling that $|c| = \|r_{k-1}^{\mathrm{FF}}\|$, we obtain $\|u_k\| = \|r_{k-1}^{\mathrm{FG}}\| / \|r_{k-1}^{\mathrm{FF}}\|$.

    Invoking the induction assumption lets us write $\omega_{k-1} = \|r_{k-1}^{\mathrm{FG}}\| / \|r_{k-2}^{\mathrm{FG}}\|$ in the notation of \eqref{eq:successive_rG}. Combining this with \eqref{eq:FFOM_FGMRES} yields $\|r_{k-1}^{\mathrm{FG}}\| / \|r_{k-1}^{\mathrm{FF}}\| = \sqrt{1 - \omega_{k-1}^2}$, and hence $\|u_k\| = \sqrt{1 - \omega_{k-1}^2}$.
    With this equality, the condition $\mu \leq \|u_k\|$ becomes $\mu \leq \sqrt{1 - \omega_{k-1}^2}$, and rearranging the terms yields $\omega_k = \frac{\mu}{\sqrt{1 - \omega_{k-1}^2}} \leq 1.$ By Proposition~\ref{prop:mu_range} together with Lemma~\ref{lem:monotonic_omega}, the assumption $\mu \leq \frac{1}{2}$ guarantees that $\omega_k < 1$. As such, the vector $y_k$ guarantees that $r_k^{\mathrm{P}} = v_k - w_k \perp \spn(W_k)$ and $\|r_k^{\mathrm{P}}\| = \mu$, and hence $\tilde r = c r_k^{\mathrm{P}}$ coincides with the FGMRES residual $r_k^{\mathrm{FG}}$, meaning that $\|r_k^{\mathrm{FG}}\| = \|r_{k-1}^{\mathrm{FF}}\|\cdot \|r_k^{\mathrm{P}}\|.$ To conclude the proof, we define the preconditioner $M_k$ to be such that $z_k = M_k^{-1} v_k = A^{-1} w_k$.
\end{proof}

The next theorem considers the more specific case when the preconditioner is GMRES($k$), guaranteeing the existence of a matrix $A$ for which the bound is attained. The proof is constructive and will be used in section~\ref{sec:numex} to illustrate the result numerically.

\begin{theorem} [Sharpness for FGMRES-GMRES($k$)] \label{thm:sharpness_gmres}
    Let $0 < \mu \leq \frac{1}{2}$, and let $m,k \in \Z^+$ such that $mk \in \Z_{<N}$. Then, for any right-hand side vector $b \in \C^N \setminus \{0\}$, there exists an invertible matrix $A \in \C^{N \times N}$ such that, for all $1 \leq j \leq m$, the FGMRES-GMRES($k$) algorithm satisfies $\|r_j^{\mathrm{P}}\| \leq \mu$ and $$\big\| r_j^{\mathrm{FG}} \big\| = \big\| r_{j-1}^{\mathrm{FF}} \big\| \cdot \big\| r_{j}^{\mathrm{P}} \big\|.$$ 
    That is, the bound \eqref{eq:FGMRESFFOMbound} holds as an equality. For this choice of linear system, the bound in \eqref{eq:successive_rG} and hence also in Theorem \ref{thm:explicit_bound} hold as equalities for FGMRES-GMRES($k$). Consequently, the bound in Theorem \ref{thm:explicit_bound} is sharp.
\end{theorem}
\begin{proof}
    We define the vectors $w_1, \dots, w_m$ as in the proof of Theorem \ref{thm:sharpness}.
    We show that we can construct a matrix $A$ such that $A z_j = w_j$ holds for all $j \leq m$ in the FGMRES-GMRES($k$) algorithm. The vectors $w_j$ are constructed such that this proves the result, as per the proof of Theorem \ref{thm:sharpness}.

    We define the vectors $v_1, \dots, v_{m+1}$ as the orthonormal basis for $\spn\rev{([b \ w_1 \ \dots \ w_m])}$ coming from the Gram–Schmidt process. For each run of the preconditioner GMRES($k$) (outer iteration $j \in \{1, \dots, m\}$), we define $k-1$ unit vectors $d_{j,1}, d_{j,2}, \dots, d_{j,k-1}$ such that all $d_{j,i}$ are orthogonal to each other and orthogonal to  $v_1, \dots, v_{m+1}$.

    \rev{For conciseness, we use the block notation $a_{j,p:q}$ to denote the matrix whose columns are $a_{j,p}, a_{j,p+1}, \dots, a_{j,q}$.} To construct the operator $A$, we let $A = Y X^*$, where
    \begin{align*}
        X & = \rev{[v_{1:m} \ d_{1,1:k-1} \ d_{2,1:k-1} \ \dots \ d_{m,1:k-1}],} \\
        Y & = \rev{[d_{1:m,1} \ d_{1,2:k-1} \ w_{1} \ d_{2,2:k-1} \ w_{2} \ \dots \ d_{m,2:k-1} \ w_{m}].}
    \end{align*}
    Since $A$ maps a column of $X$ to the corresponding column of $Y$, defining $A$ in this way is equivalent to defining it via its action on the vectors:
    \begin{align*}
        A v_j       & := d_{j,1},                                          \\
        A d_{j,i}   & := d_{j,i+1} \quad \text{(for $i = 1, \dots, k-2$)}, \\
        A d_{j,k-1} & := w_j.
    \end{align*}
    With this definition of $A$, the Krylov subspace can be expressed as
    $$ \mathcal{K}_k(A, v_j) = \spn\{v_j, d_{j,1}, d_{j,2}, \dots, d_{j,k-1}\}, $$
    and its image under $A$ is
    $$ A \mathcal{K}_k(A, v_j) = \spn\{d_{j,1}, d_{j,2}, \dots, d_{j,k-1}, w_j\}. $$
    At iteration $j$, the vectors $d_{j,i}$ are orthogonal to $v_j$ and $w_j$ by construction. As such, the GMRES($k$) solution $z_j$ satisfies \rev{$Az_j \in \spn( w_j )$}. From the proof of Theorem \ref{thm:sharpness}, $w_j$ is constructed such that $(v_j - w_j) \perp w_j$. Thus, $Az_j = w_j$, which guarantees that $\|r_j^{\mathrm{P}}\| = \mu$ and $$\big\| r_j^{\mathrm{FG}} \big\| = \big\| r_{j-1}^{\mathrm{FF}} \big\| \cdot \big\| r_{j}^{\mathrm{P}} \big\|,$$ as desired.



    To complete the matrix $A$ into an invertible matrix,
    we can extend $A$ to act, for example, as the identity on the remaining $N - mk$ dimensions (by computing an orthonormal basis to $\ker(X^*)$ and appending it to the columns of $X$ and $Y$). 
\end{proof}

We conjecture that it is possible to prove sharpness for an arbitrary sequence of matrix preconditioners. The precise statement is as follows.

\begin{conjecture}[Sharpness for matrix preconditioners]
    Let $0 < \mu \leq \frac{1}{2}$. For any right-hand side vector $b \in \C^N \setminus \{0\}$ and any choice of matrix preconditioners $(M_j)_{1\leq j < N}$, there exists an invertible matrix $A \in \C^{N \times N}$ such that, for all $1 \leq j < N$, it holds that $\|r_j^{\mathrm{P}}\| \leq \mu$ and the bound in Theorem~\ref{thm:explicit_bound} holds as an equality.
\end{conjecture}

The main technical difficulty in proving this conjecture lies in verifying the linear independence of the vectors $z_j$, which is relied upon by the proof techniques used in the proofs of Theorem~\ref{thm:sharpness} and Theorem~\ref{thm:sharpness_gmres}. For this reason, we leave this statement as an open conjecture.

\section{Numerical demonstrations}\label{sec:numex}

We now provide some numerical illustrations of the presented theory. All experiments are performed in MATLAB R2026A and the matrices and codes for reproducing all three experiments can be found on
\smallskip
\begin{center}
\url{https://github.com/NymanLauri/convergence-of-FGMRES}
\end{center}

\subsection{Illustrating sharpness for $\mu\leq 1/2$}
We first verify the sharpness of the bound derived in Theorem \ref{thm:sharpness_gmres}. We construct a test matrix $A$ of size $N=2,500$ and right-hand side vector~$b$ according to the proof of Theorem~\ref{thm:sharpness_gmres}, such that the contraction factor of the preconditioner (GMRES(100)) is $\mu = 0.5$ at each iteration. This linear system is saved in the file \texttt{sharp\_system.mat} of the GitHub repository. We run FGMRES for $m=20$ outer iterations.
Figure~\ref{fig:fgmres_sharp} shows that the computed FGMRES residual indeed matches its upper bound, confirming that this theoretical bound is sharp and cannot be improved in the general case.

\begin{figure}[htbp]
    \centering
    \begin{tikzpicture}
        \begin{semilogyaxis}[
                xlabel={Iterations},
                legend pos=south west,
                legend columns=1,
                width=0.7\linewidth,
                height=7cm,
                font=\small
            ]

            \addplot[black, only marks, thick, mark=x, mark size=2pt, mark options={fill=black, solid}]
            table[x index=0, y index=1] {fgmres_sharp.dat};
            \addlegendentry{$\mathrm{FGMRES}$ residual};

            \addplot[orange, dashed, thick]
            table[x index=0, y index=2] {fgmres_sharp.dat};
            \addlegendentry{$\mathrm{FGMRES}$ bound \eqref{eq:cheb_bound}};

        \end{semilogyaxis}
    \end{tikzpicture}
    \caption{Sharpness of FGMRES convergence bound. The matrix $A \in \mathbb{C}^{2500 \times 2500}$ was constructed such that the FGMRES-GMRES(100) bound holds as an exact equality for the right-hand side vector $b=e_1$, with an inner contraction factor $\mu=0.5$. The outer loop was run for 20 iterations. The plot clearly shows the two phases of FGMRES convergence, initially at geometric rate~$\mu$, and gradually deteriorating to~$\nu = \sin(\frac{1}{2} \arcsin(2\mu)) = \sqrt{\mu}$.}
    \label{fig:fgmres_sharp}
\end{figure}
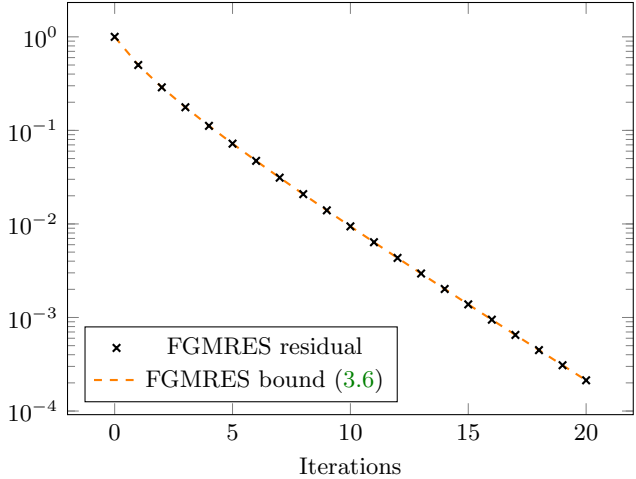

\subsection{Illustrating stagnation for $\mu>1/2$}

We now provide an example for which FGMRES stagnates. This is illustrated in Figure \ref{fig:fgmres_stalling} where the residual norm stagnates after six iterations for $\mu=0.55$. The relative residual at stagnation is roughly $0.2$. The corresponding linear system is saved in the file \texttt{stagnating\_system.mat} of the GitHub repository.

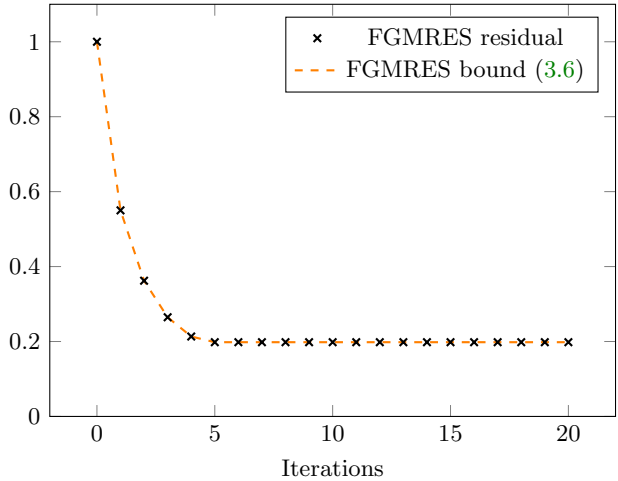
\begin{figure}[htbp]
    \centering
    \begin{tikzpicture}
        \begin{axis}[
                xlabel={Iterations},
                ymin=0,
                legend pos=north east,
                legend columns=1,
                width=0.7\linewidth,
                height=7cm,
                font=\small
            ]

            \addplot[black, only marks, thick, mark=x, mark size=2pt, mark options={fill=black, solid}]
            table[x index=0, y index=1] {fgmres_stalling_sharp.dat};
            \addlegendentry{$\mathrm{FGMRES}$ residual};

            \addplot[orange, dashed, thick]
            table[x index=0, y index=2] {fgmres_stalling_sharp.dat};
            \addlegendentry{$\mathrm{FGMRES}$ bound \eqref{eq:cheb_bound}};

        \end{axis}
    \end{tikzpicture}
    \caption{Stagnation of FGMRES for $\mu=0.55$. The matrix $A \in \mathbb{C}^{2500 \times 2500}$ was constructed such that the FGMRES-GMRES(100) bound holds as an exact equality for the right-hand side $b=e_1$, with an inner contraction factor $\mu=0.55$. The outer loop was run for 20 iterations. The plot shows that the residual norm stagnates at iteration~6.}
    \label{fig:fgmres_stalling}
\end{figure}

\subsection{Illustration on a finite element matrix}

Our third test in Figure~\ref{fig:fgmres_conv} illustrates the bound for a matrix coming from the finite element discretization of a 3D Poisson problem (\url{https://sparse.tamu.edu/FEMLAB/poisson3Db}).
The matrix $A$ is real, nonsymmetric, and sparse with $N=85,623$ and $2,374,949$ nonzeros. The preconditioner is GMRES with a target residual reduction of $\mu=0.1$. The maximal number of GMRES iterations across the nine outer FGMRES iterations is 68 (the individual iteration counts are 6, 27, 66, 68, 56, 60, 45, 54, 51). We find that our new bound~\eqref{eq:cheb_bound} still follows the actual convergence closely.

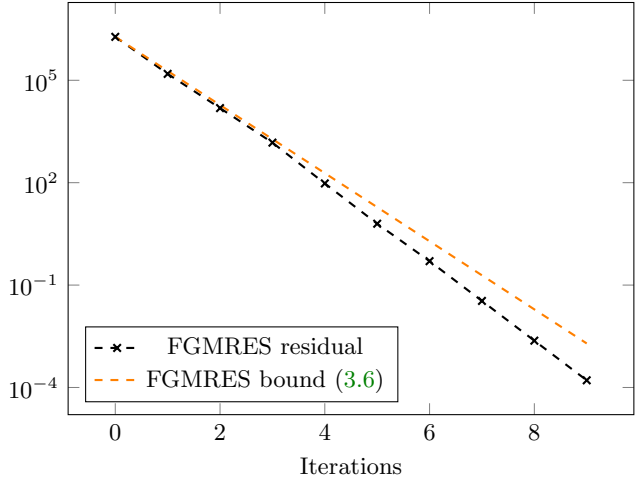
\begin{figure}[htbp]
    \centering
    \begin{tikzpicture}
        \begin{semilogyaxis}[
                xlabel={Iterations},
                legend pos=south west,
                legend columns=1,
                width=0.7\linewidth,
                height=7cm,
                font=\small
            ]

            \addplot[black, dashed, thick, mark=x, mark size=2pt, mark options={fill=black, solid}]
            table[x index=0, y index=1] {fgmres_results.dat};
            \addlegendentry{$\mathrm{FGMRES}$ residual};

            \addplot[orange, dashed, thick]
            table[x index=0, y index=2] {fgmres_results.dat};
            \addlegendentry{$\mathrm{FGMRES}$ bound \eqref{eq:cheb_bound}};

        \end{semilogyaxis}
    \end{tikzpicture}
    \caption{Convergence of FGMRES-GMRES on a 3D Poisson matrix. There is good agreement between the computed residuals and the upper bound given in~\eqref{eq:cheb_bound}.}
    \label{fig:fgmres_conv}
\end{figure}


    

\section{Conclusions}\label{sec:concl}
We have derived an attainable upper bound on the FGMRES residual. This bound shows that FGMRES convergence is governed by two phases, an initial phase of geometric convergence at rate~$\mu$, followed by a phase where convergence deteriorates to $\nu = \sin(\frac{1}{2} \arcsin(2\mu))$ or stagnates. 

One might wonder what are the practical consequences of such a finding? First, our results show that when a sufficiently good preconditioner with $\mu\ll 1/2$ is available, it is safe to assume that FGMRES will convergence roughly at rate~$\mu$ throughout. See also Table~\ref{tab:mu}. If the preconditioner is a tunable method where a target $\mu$ can be chosen freely   (e.g., with GMRES without imposing an upper limit on the number of iterations), and if it is known roughly how the cost of the preconditioner increases as a function of~$\mu$, then our new bound will pave the way for a more informed, less heuristical, choice of the tolerances. 

Another consequence of our work is that FGMRES is guaranteed to converge to the exact solution of $Ax=b$ if and only if $\mu\leq 1/2$ (assuming no other knowledge about the problem). We believe this is the first such result. It is interesting to compare this to restarted GMRES. If GMRES is restarted with a fixed target relative residual decrease of $\mu$ per cycle, then the overall residual decreases like $O(\mu^j)$. This is better than the worst case convergence of FGMRES given by our bound (which is sharp). On the other hand, it has been observed repeatedly (e.g., in~\cite{saad1993flexible,simoncini2002flexible,guttel2025stabilizing}), that FGMRES-GMRES($k$) typically outperforms restarted GMRES($k$). This may be partially explained by the fact that each right-hand side vector $v_j$ handed to the \mbox{FGMRES}  preconditioner $M_j$ is effectively deflated against the previous vectors $v_1,\ldots,v_{j-1}$; see~\cite{simoncini2002flexible}. Hence it may be more adequate to compare FGMRES-GMRES($k$) against GMRES-DR($k,j-1$), that is, GMRES with deflated restarting~\cite{Morgan2002} using $j-1$ deflation vectors at the $j$-th cycle. It appears that spectrally unaware convergence bounds cannot explain these deflation effects, and further research is needed here.

Our key assumption throughout the paper was~\eqref{eq:assumemu}. There are situations where one has limited control over the residual decrease $\mu$ of the preconditioner. An example is the fastGMRES method in \cite{guttel2025stabilizing}. Therein, the preconditioner is a version of GMRES utilizing a nonorthogonal basis and randomized sketching, and the maximal number of iterations (and hence the residual norm) is limited by the exponentially growing condition number of the Krylov basis. Our global convergence results do not apply in such a situation, or they may be too pessimistic if $\mu$ is chosen as an upper bound for all FGMRES cycles. However, most of the recursive results in section~\ref{sec:recur} are still valid. We hope to extend our analysis to a setting that can accommodate for a small number of FGMRES cycles with large~$\mu$ (even $\geq 1$) while still predicting convergence.

\section*{Acknowledgments}
Both authors acknowledge funding from the UK's Engineering and Physical Sciences Research Council (EPSRC grant EP/Z533786/1). SG~is supported by Royal Society Industry Fellowship IF/R1/231032.  We are grateful to Karl Meerbergen and John Pearson for useful discussions.

\bibliographystyle{plain}
\bibliography{refs}

\end{document}